\documentclass{amsart}
\usepackage{chngcntr}
\usepackage{apptools}
\usepackage{color}
\usepackage{amsthm,amssymb,verbatim}
\usepackage{mathtools}
\usepackage{graphicx}
\usepackage{enumerate}
\usepackage{enumitem}
\usepackage{chngcntr}
\usepackage{apptools}
\usepackage{color}
\usepackage{amsthm,amssymb,verbatim}
\usepackage{graphicx}
\usepackage{enumerate}

 \newcommand{\eps}{\epsilon}

\usepackage{amsthm}

\input epsf
\def\begfig {
\begin{figure}
\small }
\def\endfig {
\normalsize
\end{figure}
}

    \newtheorem{theorem}    {Theorem}   
    \newtheorem{lemma}      [theorem]       {Lemma}
    
    \newtheorem{proposition}       [theorem]       {Proposition}
    
    \newtheorem*{claim}{Claim}
    \newtheorem*{theorem*}{Theorem}
    \theoremstyle{definition}
    \newtheorem{definition}  [theorem] {Definition}
     
    \theoremstyle{definition}
    \newtheorem{remark}   [theorem]       {Remark}

\title[Translators asymptotic to cylinders]{Translators asymptotic to cylinders}
\author{Or Hershkovits}
\thanks{The author was partially supported by an AMS-Simons Travel Grant}
\address{Department of Mathematics\\ Stanford University\\ Stanford, CA 94305}
\email{orher@stanford.edu}
\subjclass[2010]{Primary 53C44; Secondary 49Q20.}
\subjclass[2010]{Primary 53C44; Secondary 49Q20.}
\usepackage{hyperref}
\usepackage[alphabetic, msc-links, backrefs]{amsrefs}
\begin{document}
\maketitle
\begin{abstract}
We show that the Bowl soliton in $\mathbb{R}^3$ is the unique translating solutions of the mean curvature flow which has the family of shrinking cylinders as an asymptotic shrinker at $-\infty$. As an application, we show that for a generic mean curvature flow,  all (non-static) translating limit flows are the bowl soliton. The crucial point is that we do not make any global convexity assumption, while as the same time, the asymptotic requirement is very weak.  
 \end{abstract}
 
 \vspace{10 mm}
 
Translating solutions to the mean curvature flow (MCF) play a key role in the analysis of singularity formation of the flow. In particular, as far as the author knows, self-shrinking solutions and translating solutions are the only kind of limit flows that have been shown to arise when blowing up a singularity. 
In the mean convex setting, an important paper of Haslhofer \cite{Haslhofer_bowl} (which in turn, uses ideas from a very important paper of Brendle analyzing the analogues situation in Ricci flow \cite{Brendle_sym}) shows that the only translator that arises as a blow up of a compact mean-convex MCF in $\mathbb{R}^3$ is the bowl soliton - a graphical, rotationally symmetric translating solution which, up to logarithmic terms, is asymptotic to a paraboloid. More recently, a ground-breaking paper of Brendle and Choi \cite{BC}  showed that the bowl solution is the unique non-shrinking blow-up limit that arises in compact mean-convex MCF. Both results follow from a combination of the theory of mean convex flows \cite{white_size,white_nature,HK_mean_convex} and  characterization results: That the Bowl is the unique strictly convex, non-collapsed translator in $\mathbb{R}^3$ in \cite{Haslhofer_bowl} and that it's the unique strictly convex, ancient, non-collapsed non-compact solution to the MCF in $\mathbb{R}^3$ in \cite{BC}. The fact that information of ``being ancient'' is successfully used in the latter paper is what makes it so remarkable.  

There are other uniqueness result for the bowl solution. The work of Clutterback, Schnurer and Schulze \cite{CSS} implies that it is the unique rotationally symmetric translating disk. An import result of Wang \cite{Wang} implies that it's the unique entire convex translating graph. In a recent paper, Spruck and Xiao \cite{SX} have shown that the convexity assumption of \cite{Wang} can be weakened to mean convexity. 

Abandoning the global mean convex world, a beautiful paper of Matrin,Savas-Halilaj and Smoczyk \cite{MHS} uses the moving plane method to show that a translator which has an end which is strongly asymptotic to  the bowl solution is the bowl. More precisely, if a translator $M$ has an end in which it can be written as a graph of a function $g:\mathbb{R}^2-B(0,R)\rightarrow \mathbb{R}$ satisfying
\begin{equation}\label{CHM}
g(x,y)=\frac{1}{2}(x^2+y^2)-\log(\sqrt{x^{2}+y^2})+O\Big(\frac{1}{\sqrt{x^2+y^2}}\Big)
\end{equation} 
then $M$ is the bowl-soliton.

\vspace{5 mm}

In this paper, we classify all translators that can appear as blow-up flows to generic MCF i.e. to MCF encountering only multiplicity one cylinders and spheres as tangent flows. Most importantly, \textit{we do not impose any global (or even local) convexity assumption around the singularity}.
\begin{theorem}\label{main_thm_b}
Let $M_t\subseteq \mathbb{R}^3$ be a MCF encountering only cylindrical singularities. Then every (non-static) translating limit flow to $M_t$ is the bowl soliton.
\end{theorem}
\begin{remark}
Forcing all tangent flows  to be cylindrical seemingly a tractable task, in light of the progress in Colding and Minicozzi's generic mean curvature program \cite{CM}, and in light of Brendle's result about genus-$0$ self shrinkers \cite{Brendle_self_shrink}. On the other hand, the recent papers of Bamler and Kleiner \cite{BK} and Hershkovits and White \cite{HW} indicate that information about \textit{all limit flows}, and not just tangent flows, is very helpful in addressing uniqueness questions. Theorem \ref{main_thm_b} allows one to convert information about tangent flows into information about other limit flows, which can be very relevant in light of the discussion above.
\end{remark} 

\vspace{5 mm}

Similarly to \cite{Haslhofer_bowl} and \cite{BC},  in our case too, Theorem \ref{main_thm_b} is a corollary of a classification result.
        
\begin{theorem}\label{main_thm}
Let $M_t\subseteq \mathbb{R}^3$ be a properly embedded translating solution to the MCF, i.e.
\[
H=\langle \tau,\nu \rangle,
\]
where $\tau=\frac{\partial}{\partial z}$. Then if the asymptotic shrinker of $M$ at $-\infty$ is a cylinder, then $M$ is the Bowl soliton. More precisely, letting $M_t:=M+t\tau$ be the evolution by MCF of $M$, then if $M_t/\sqrt{-t}\rightarrow S^{1}(\sqrt{2})\times \mathbb{R}$ smoothly and locally as $t\rightarrow -\infty$, then $M$ is the bowl soliton.   
\end{theorem}
\noindent From the static point of view there are two key points:
\begin{enumerate}[label=\Alph*.]
\item No global convexity assumptions are imposed.
\item The asymptotic assumption is very weak.
\end{enumerate}
In particular, Theorem \ref{main_thm} implies Theorem \ref{main_thm_b}, and generalizes the results of Haslhofer \cite{Haslhofer_bowl} and the one of Matrin,Savas-Halilaj and Smoczyk \cite{MHS}. In fact, the proof works by promoting the weak asymptotic assumptions about the asymptotic shrinker into the stronger ones of \cite{MHS}, from which the result follow.

\vspace{5mm}

Our proof relies on the result of \cite{MHS} (at the final step) and, crucially,  on the recent neck-improvement theorem of \cite{BC} (in the first step). In fact, we will need a simpler version of the neck improvement lemma, in which the axis of symmetry is fixed, up to translations.

Let us briefly discuss the proof. For $(x_0,y_0)\in \mathbb{R}^2$ Let $J_{(x_0,y_0)}$ be the rotation vector field around a translate of the $z$  axis, passing through $(x_0,y_0,0)$,  i.e.
\[
J_{(x_0,y_0)}=(x-x_0)\frac{\partial}{\partial y}-(y-y_0)\frac{\partial}{\partial x}
\]
and let 
\begin{equation}
u_{(x_0,y_0)}=\langle J_{(x_0,y_0)}, \nu \rangle
\end{equation}
be the rotation function. By a repeated use of (a primitive version of) the neck improvement theorem of \cite{BC}, we get that, as $z\rightarrow \infty$, one can choose $(x_0,t_0)$ (depending on $z$) such that $u_{(x_0,t_0)}$ decays rapidly. Unlike \cite{BC} and \cite{Haslhofer_bowl}, we can not argue by applying maximum principle arguments for $u/H$, as $H$ is not globally positive. Rather, the rapid decay of the rotation functions implies a rapid decay for the change of the symmetry axis. Thus one can choose a \textit{fixed} $z$-axis, passing through some $(x_0,y_0)\in \mathbb{R}^2$ independent of $z$ such that the cylindrical end of $M$ can be expressed in cylindrical co-ordinates as a graph of a rotationally symmetric part $f(z)$ plus a (small) angle dependent part $g(z,\theta)$.  Promoting the decay of $u_{(x_0,y_0)}$ to higher derivatives decay, we see that  the rotationally symmetric part $f$ satisfies the rotationally symmetric translator equation in cylindrical co-ordinates , up to small errors. Using ODE arguments much in line with \cite{CSS}, we see that $f$ is increasing with $z$ (when $z$ is large enough), and thus, can be written as a vertical graph over the $xy$ plane (minus a ball). The function defining the graph satisfies the graphical rotationally symmetric translator equation, up to small errors. Arguing like \cite{CSS}, we then show that $f$ satisfies the growth condition \eqref{CHM}. Thus, the end of $M$, which is the cylindrical graph of $f(z)+g(z,\theta)$ is itself a graph over the $xy$ plane (minus a ball), satisfying \eqref{CHM}. The proof is then completed by appealing to \cite{MHS}. 
\begin{remark}
The argument of \cite{Haslhofer_bowl} can also be used to get an asymptotic decay of the rotation function. The order of that decay, however, is insufficient for our approach.
\end{remark}
The organization of the paper is as follows: In section 2 we adapt Brendle-Choi neck improvement theorem to a simpler version, in which the axis of symmetry can move, but not rotate. In Section 3, which is the main one of this paper, we prove Theorem \ref{main_thm}. In Section 4, we show how Theorem \ref{main_thm} implies Theorem \ref{main_thm_b}.
\subsection*{Acknowledgments}
The author would like to thank Brian White for sharing his knowledge and insights about the mean curvature flow with him.
\section{Brendle-Choi neck improvement theorem}
We present an easier variant of Brendle-Choi's neck improvement theorem \cite[Theorem 4.4]{BC}. In our situation the axis of approximate cylindricality and rotational symmetry can move, but not tilt. In return, the improved rotational symmetry is around an axis parallel to the original one.
\begin{definition}
A point $(\bar{x},\bar{t})$ in a MCF $M_t$ is called \textbf{$\eps$-vertically cylindrical} if there exist a  $\lambda>0$ and $v\in S^{1}(\sqrt{2})\times \{0\}$ such that $\lambda(M_{\lambda^{-2}t+\bar{t}}-\bar{x})+v$ is $\eps$ close in the $C^{10}$ norm on $P((0,0),10)$ to the family of shrinking cylinders with $\tau$-axis $S^1(\sqrt{2(1-t)})\times \mathbb{R}$. Here $P((x,t),r)$ is the parabolic ball 
\begin{equation}
P((x,t),r)=B(x,t)\times [t-r^2,t].
\end{equation}
\end{definition} 
\begin{definition}
A point $(\bar{x},\bar{t})$ in a MCF $M_t$ is called \textbf{$\eps$-vertically symmetric} if there exists some $(x_0,y_0)$ such that $|u_{(x_0,y_0)}H|\leq \eps$ and $H>0$ on $P((\bar{x},\bar{t}), 10H(\bar{x},\bar{t})^{-1})$.
\end{definition}
We state a weaker version of Brendle Choi's neck improvement theorem \cite[Theorem 4.4]{BC}.
\begin{theorem}[Non-tilting Brendle-Choi's neck improvement theorem]\label{neck_improv}
There exist $L<\infty$ and $\eps_1>0$ such that with the following significance: Suppose $M_t$ is MCF and $(\bar{x},\bar{t})$  is a point such that every point in $P((\bar{x},\bar{t}),LH^{-1}(\bar{x},\bar{t}))$ is $\eps_1$-vertically cylindrical and is $\eps$-vertically symmetric with $\eps\leq \eps_1$. Then $(\bar{x},\bar{t})$ is $\eps/2$-vertically symmetric. 
\end{theorem}
\begin{remark}
Theorem \ref{neck_improv} is proved exactly as \cite[Theorem 4.4]{BC}. The fact that the rotation axis is vertical allows one to choose $a_1=b_1=0$ in their proof. As a result, In the analysis of the mode $m=1$, one gets the corresponding estimates for the first derivative of $v_1$ and $w_1$ (and not their second derivatives), which allows one to take $A_1=B_1=0$, and so the improved axis of rotation does not tilt. 
\end{remark}
\section{Proof of the classification theorem}
On this section, we prove Theorem \ref{main_thm}, regarding the Bowl soliton being the unique translator asymptotic to a cylinder. As a first step, we get an improved approximate rotational symmetry as $z\rightarrow \infty$.
\begin{lemma}\label{decay_lemma}
There exists a $\Lambda$ such that if $\bar{p}=(\bar{x},\bar{y},\bar{z})\in M$ with $\bar{z} \geq \Lambda$ and $\bar{x}^2+\bar{y}^2\leq 1000\bar{z}$ then there exists $(x_0,y_0)\in \mathbb{R}^2$ such that
\[
|u_{(x_0,y_0)}| \leq \frac{C}{z^{100}}
\]
on $\{(x,y,z)\in \mathbb{R}^3\;|\; x^2+y^2\leq 1000\bar{z},\;-10\leq z-\bar{z} \leq 0\}$.
\end{lemma}
\begin{remark}
The appearance of expressions like $\bar{x}^2+\bar{y}^2\leq 1000\bar{z}$ is technical, and can be ignored. Knowing that the asymptotic shrinker at $-\infty$ is a cylinder doesn't rule out, a priori, that $M$ has other ends for large $z$. We do know, however, that such parts have to be far away, and so our statement is about the end of $M$ that ``looks like'' a cylinder. 
\end{remark}
\begin{proof}
The proof of this lemma is similar to the proof of \cite[Proposition 5.3]{BC}. Let $M_t=M+t\tau$ and let $\eps_1$ and $L$ be the constants from the non tilting neck improvement theorem, Theorem \ref{neck_improv}.  Then there exists a $\Lambda$ such that if $p=(x,y,z)\in M_t$ satisfies $z-t \geq \Lambda$  and $x^2+y^2\leq 1000(z-t)$, then 
\begin{equation}\label{H_growth}
H(p)(z-t)\geq 2(1-2^{-1/400})^{-1}L
\end{equation}
and $(p,t)$ is $\eps_1$-vertically symmetric and is $\eps_1$-vertically cylindrical. 
\begin{claim}
If $p=(x,y,z)\in M_t$ satisfies $z-t \geq 2^{j/400}\Lambda$ and $x^2+y^2 \leq 1000(z-t)$ then $(p,t)$ is $2^{-j}\eps_1$-vertically symmetric. 
\end{claim} 
\begin{proof}[Proof of claim]
We argue by induction and contradiction. When $j=0$, this holds by our assumption. Now, assume that this does not hold for $j$. Then there exists a point $(x,y,z)=p\in M_t$ with $z-t\geq  2^{j/400}\Lambda$ and $x^2+y^2 \leq 1000(z-t)$ which is not $2^{-j}\eps_1$-vertically symmetric. But then, by the neck improvement theorem, there is a point in space time $(p',t')$ with 
\[
t-L^2H(p,t)^{-2}\leq t'\leq t,\;\;\;|p'-p|\leq LH(p,t)^{-1},
\] 
which is not $2^{j-1}\eps_1$-vertically symmetric. Setting $p'=(x',y',z')$ we obtain, by the induction hypothesis and by \eqref{H_growth}
\begin{align}
z-t &\leq  z-t' \leq |p-p'|+(z'-t') \leq LH(p,t)^{-1}+2^{(j-1)/400}\Lambda \\
&\leq \big(\frac{1-2^{-1/400}}{2}+2^{-1/400}\big)(z-t)<z-t.
\end{align}
which is a contradiction. 
\end{proof}
To complete the proof of the lemma, consider $M$ as a static object. For each $z$ we can find a $j$ such that $2^{j/400}\Lambda<z\leq 2^{(j+1)/400}\Lambda$. The previous claim implies that there exists some $(x_0,y_0)$ such that 
\[
|u_{(x_0,y_0)}H| \leq 2^{-j}\eps_1\leq \frac{2\Lambda^{400}}{z^{400}}\eps_1.
\]
As $H\cong \frac{1}{\sqrt{z}}$, the result follows.
\end{proof}
\begin{remark}
Arguing as in \cite[Theorem 5.4]{BC}, one can replace \cite[Proposition 5.3]{BC} with Lemma \ref{decay_lemma} to obtain Haslhofer's theorem \cite{Haslhofer_bowl}:  if $M$ is globally mean convex and non-collapsed translator in $\mathbb{R}^3$ then it is the bowl soliton. The proof of the main theorem in \cite{BC} actually builds upon Haslhofer's result, the original proof of which involves other (related) ideas. Lemma \ref{decay_lemma} thus gives a direct proof, along the lines of \cite{BC} to Haslhofer's result, making the argument in \cite{BC} self contained. Of course, the main theorem of \cite{Haslhofer_bowl} will also follow once we establish Theorem \ref{main_thm}.
\end{remark}
\vspace{5mm}
Now, the point in $(x_0,y_0)$ of Lemma \ref{decay_lemma} depends on $z$. However, it clearly follows that if $|z-z'| \leq 10$, then $|(x_0,y_0)^z-(x_0,y_0)^{z'}| \leq \frac{C}{z^{99}}$, where $(x_0,y_0)^z$ is the $(x_0,y_0)$ that the lemma produces for height $z$. Thus, as 
\[
\int_{\bar{z}}^{\infty}\frac{1}{z^{99}}dz=\frac{C}{\bar{z}^{98}}
\]
the $(x_0,y_0)^{z}$ converge as $z\rightarrow \infty$ and we get the following rotational decay about a fixed axis:  
\begin{lemma}\label{decay_fixed}
There exists a $\Lambda$ and $(x_0,y_0)$ such that if $\bar{p}=(\bar{x},\bar{y},\bar{z})\in M$ with $\bar{z} \geq \Lambda$ and $\bar{x}^2+\bar{y}^2\leq 1000\bar{z}$ then 
\[
|u_{(x_0,y_0)}(\bar{p})| \leq \frac{C}{z^{98}}.
\]
\end{lemma}  
From now on, we assume without loss of generality that $(x_0,y_0)=(0,0)$ and we write $u=u_{0,0}$.
For $z\geq \Lambda$ , we can write the cylindrical end of $M$ in cylindrical co-ordinates around the $z$-axis:
\[
M\cap \{z\geq \Lambda\}=\{(r(z,\theta)\cos\theta,r(z,\theta)\sin\theta,z)\;|\;z\geq \Lambda,\;\theta\in S^1\},
\]
where $r=\sqrt{2z}+o(\sqrt{2z})$.
\begin{proposition}
In the co-ordinates $z,\theta$, the function $\frac{\partial r}{\partial\theta}$ satisfies the estimate
\begin{equation}
\frac{\partial^{i+j}}{\partial \theta^i\partial z^j}\frac{\partial r}{\partial\theta} \leq \frac{C}{z^{50}},
\end{equation} 
for $z$ large enough, and for $i+j\leq 10$.
\end{proposition}
\begin{proof}
Considering the flow $M_t$, the asymptotic assumption implies that 
\[
M_{-t_0t}/\sqrt{-t_0}\xrightarrow{t_0\rightarrow -\infty}S^1(\sqrt{-2t})\times \mathbb{R}.
\]
in $C^{20}(B(0,100)\times [-100,-1])$. Setting
\[
\tilde{u}(z,\theta,t)=\frac{u(\sqrt{-t_0}(z-t_0t),\theta,-t_0t)}{\sqrt{-t_0}},\;\;\; \tilde{r}(z,\theta,t)=\frac{r(\sqrt{-t_0}(z-t_0t),\theta,-t_0t)}{\sqrt{-t_0}},
\]
the previous lemma implies that $|\tilde{u}|\leq Ct_0^{-80.5}$ in $B(0,100)\times [-100,-1])$. As $\tilde{u}$ satisfies the linearized MCF equation 
\[
\partial_t \tilde{u}=\Delta \tilde{u}+|A|^2\tilde{u}.
\]
on a open manifold which is very close in $C^{20}$ to a family of shrinking cylinders, it follows from standard parabolic estimates that 
\begin{equation}
\frac{\partial^{i+j}}{\partial \theta^i\partial z^j}\frac{\partial \tilde{u}}{\partial\theta} \leq \frac{C}{t_0^{80.5}},
\end{equation}
whenever $i+j\leq 10$.
on $B(0,25)\times [-25,-1]$, and, setting $\tilde{u}(z,\theta)=\tilde{u}(z,\theta,-1)$ one gets the same estimate without the time dependence.
Now, the relation between $\tilde{u}$ and $\tilde{r}$ is given by
\begin{equation}
\tilde{u}=\frac{1}{\sqrt{1+\tilde{r}^{-2}\Big(\frac{\partial \tilde{r}}{\partial \theta}\Big)^2+\Big(\frac{\partial \tilde{r}}{\partial z}\Big)^2}}\frac{\partial \tilde{r}}{\partial \theta}.
\end{equation}
As $\frac{\partial \tilde{r}}{\partial \theta}$ and $\frac{\partial \tilde{r}}{\partial z}$ are as small as we want in $C^{20}$, we get the same $C^{10}$ estimate for $\frac{\partial \tilde{r}}{\partial \theta}$ as we do for $\tilde{u}$ (in the $\theta,z$ co-ordinates). The result now follows from scaling back to $r$.
\end{proof}
The following proposition is an immediate consequence of the convergence to the cylinders and scaling.
\begin{proposition}
\[
\Big|\frac{\partial^i r}{\partial z^i}\Big| \leq Cz^{i+0.5}
\]
\end{proposition}
By the previous two propositions, it follows that for $z$ large enough, we can write $r(z,\theta)=f(z)+g(z,\theta)$, where 
\begin{equation}\label{g_est}
||g||_{C^{10}}\leq \frac{C}{z^{50}}.
\end{equation}
Now, by \cite[Eq. 2.1]{Knopf-Gang}, $r$ satisfies
\begin{equation}
-\frac{\partial r}{\partial z}=\frac{\Big(1+\Big(\frac{1}{r}\frac{\partial r}{\partial \theta}\Big)^2\Big)\frac{\partial^2 r}{\partial z^2}+\frac{1+(\frac{\partial r}{\partial z})^2}{r^2}\frac{\partial^2 r}{\partial \theta^2}-2\frac{\frac{\partial r}{\partial z}(\frac{\partial r}{\partial \theta})^2}{r^3}\frac{\partial^2 r}{\partial \theta \partial r}-\frac{(\frac{\partial r}{\partial \theta})^2}{r^3}}{1+\Big(\frac{\partial r}{\partial z}\Big)^2+\Big(\frac{1}{r}\frac{\partial r}{\partial \theta}\Big)^2}-\frac{1}{r},
\end{equation} 
and so, in light of the above, evaluating the equation in $\theta=\theta_0$, the function $f(z)$ satisfies the following differential equation:
\begin{equation}\label{est_1}
-f_z=\frac{(1+a(z))f_{zz}}{1+f_{z}^2}-\frac{1}{f}+b(z)
\end{equation}
with
\begin{equation}\label{est_2}
|a(z)|+|b(z)| \leq \frac{C}{z^{10}},
\end{equation}
and $f(z)=\sqrt{2z}+o(\sqrt{z})$.
We are now in the position to use the ODE methods of Clutterback-Schnurer-Schulze \cite{CSS} to first show that the cylindrical end of $M$ is a graph over the $xy$ plane, and then study its asymptotic properties.
\begin{lemma}
For $z$ large enough, $f$ is increasing.
\end{lemma}
\begin{proof}
As $f\rightarrow \infty$ when $z\rightarrow \infty$, there must be a point $z_0$ with $f'(z_0)> 0$, and  where estimates \eqref{est_1},\eqref{est_2} are valid.
But then, if $f$ has a critical point $z_1$, with $z_1>z_0$ we get from \eqref{est_1} and \eqref{est_2} that $f_{zz}(z_1)>0$ which is a contradiction.  
\end{proof}
The previous lemma shows that the cylindrical graph of $f$ is in fact a (rotationally symmetric) graph over the $xy$ plane. Letting $z(s)$ be the inverse of $f(z)$ (which is defined when $z$ is large enough), we have
\[
f_z=\frac{1}{z_s},\;\;\;\;f_{zz}=-\frac{z_{ss}}{z_s^3}.
\] 
Thus, $z$ satisfies the equation
\begin{equation}
1=\frac{(1+\alpha(s))z_{ss}}{1+z_s^2}+\Big(\frac{1}{s}+\beta(s)\Big)z_s,
\end{equation}
where 
\[
|\alpha(s)|+|\beta(s)|\leq \frac{C}{s^{20}}.
\]
Setting $\phi(s)=z_s(s)$ we have:
\begin{equation}
\phi'=(1+\phi^2)\Big(1+\gamma(s) -\big(\frac{1}{s}+\delta(s)\big)\phi\Big)
\end{equation}
where 
\[
|\gamma(s)|+|\delta(s)|\leq \frac{C}{s^{10}}.
\]
We can now argue as in  \cite[Lemma 2.1]{CSS} to obtain the following asymptotics.
\begin{lemma}\label{growth_lemma}
\begin{equation}\label{der_et}
\phi(s)=s-\frac{1}{s}+O(s^{-2})
\end{equation}
\end{lemma}
Before proving this lemma, lets see how it implies the main theorem
\begin{proof}[Proof of thm. \ref{main_thm}]
Integrating \eqref{der_et} we obtain that 
\begin{equation}\label{graph_growth}
f(s)=c+s^2/2-\log(s)+O(1/s),
\end{equation}
for some $c>0$.
Together with equation \eqref{g_est}, we see that for every $C<\infty$ there exists $\Lambda<\infty$ such that
\[
M\cap \{z \geq \Lambda\}\cap \{|(x,y)|\leq C\sqrt{z}\}
\]
is a graph of a function 
\[
h:\{z=0\}-B_R\rightarrow \mathbb{R}
\]
satisfying
\begin{equation}
h(s)=c+\frac{s^2}{2}-\log(s)+O\Big(\frac{1}{s}\Big).
\end{equation}
It follows from the proof of theorem B of a paper by Martin, Savas-Halilaj and Smoczyk \cite{MHS}\ that $M\cap \{z \geq \Lambda\}\cap \{|(x,y)|\leq C\sqrt{z}\}$ is rotationally symmetric \footnote{While the assumptions of theorem B in \cite{MHS} are different, they are only used at the beginning of the proof to conclude the asymptotic graph structure with asymptotics as in \eqref{graph_growth}.}. Since the rotation function $u$ satisfies an elliptic PDE on $M$, it follows that $u=0$ everywhere. Thus $M$ is rotationally symmetric. By \cite{CSS}, a rotationally symmetric translators is either the Bowl or the translating catenoid (see also \cite{MHS}). Note that the translating catenoid, however, has a multiplicity two asymptotic cylinder at $-\infty$. Consequently, $M$ is the bowl soliton.  
\end{proof}
\begin{proof}[Proof of Lemma \ref{growth_lemma}]
As was mentioned above, this follows closely the proof of \cite[Lemma 2.1]{CSS}, where the same result is stated under the assumption $\gamma(s)=\delta(s)=0$. 
for every function $\psi:\mathbb{R}\rightarrow \mathbb{R}$, set
\[
I[\psi]=(1+\psi^2)\Big(1+\gamma(s) -\big(\frac{1}{s}+\delta(s)\big)\psi\Big).
\]
\begin{claim}
There exists $s_0$ such that for $s>s_0$ 
\begin{equation}\label{phi_lin_growth_up}
\phi(s) \leq s+\frac{1}{s^{9}}
\end{equation}
\end{claim}
\begin{proof}
If $\phi'(s) \leq 0$ for all $s$ large, then $\phi$ is bounded from above, and so the quantity $\Big(1+\gamma(s) -\big(\frac{1}{s}+\delta(s)\big)\phi\Big)$ becomes positive for $s$ large enough, contradicting $\phi' \leq 0$. Thus there exists some $s_0$ such that $\phi'(s_0)>0$. Taking $s>s_0$, there are two cases: (i) if $\phi'(s)>0$ then 
\[
\big(\frac{1}{s}+\delta(s)\big)\phi \leq 1+\gamma(s)
\] 
which gives \eqref{phi_lin_growth_up}. (ii) If $\phi'(s)<0$ then letting $s'<s$ be such that $\phi'<0$ on $(s',s]$ and $\phi'(s')=0$ we get
\[
\phi(s)<\phi(s') \leq s'+\frac{1}{s'^{9}} \leq s+\frac{1}{s^9}.
\] 
\end{proof}
\begin{claim}
For every $\eps>0$ there exists $s_0$ such that for $s>s_0$
\begin{equation}
\phi(s) \geq (1-\eps)s
\end{equation}
\end{claim}
\begin{proof}
Note that $v=(1-\eps)s$ satisfies $v'\leq I[v]$ for $s$ large enough. Thus the claim would follow if we can show that for every $s_0$ there exists $s_1>s_0$ such that $\phi(s_1) \geq (1-\eps)s_1$. But if this is false, then
\[
\phi'(s)\geq(1+\phi^2)\Big(1+\gamma(s) -\big(\frac{1}{s}+\delta(s)\big)(1-\eps)s\Big)\geq \frac{\eps(1+\phi^2)}{2},
\]
for $s$ large enough, which will result in blow up in finite $s$. this is a contradiction. 
\end{proof}
We've concluded that $\phi(s)$ can be written as $\phi(s)=s+\psi(s)$, where $\psi=o(s)$, $\psi(s)\leq \frac{1}{s^9}$ and satisfies the equation:
\begin{equation}\label{psi_eq}
\psi' =(1+(\psi+s)^2)(\gamma(s)-\big(\frac{1}{s}+\delta(s)\big)\psi)-1
\end{equation}
where $|\gamma(s)|+|\delta(s)|\leq \frac{C}{s^9}$ (the function $\gamma$ has changed in this line).
\begin{claim}
$\psi(s)\rightarrow 0$ as $s\rightarrow \infty$.
\end{claim}
\begin{proof}
Let $\eps>0$. If $\psi<-\eps$ then, since $\psi$ is sub-linear, we have, by \eqref{psi_eq}, that 
\[
\psi' \geq \frac{\eps s}{2},\;\;s\;\textrm{large enough}
\]
Thus, there exists some $s_0$ such that for $s>s_0$, $\psi(s)>-\eps$.
\end{proof}
Now, set $\lambda(s)=s\psi(s)$, and observe that $\lambda$ is sub-linear and satisfies,
\begin{equation}\label{lambda_eq_ex}
\lambda' =s(1+(\psi+s)^2)(\gamma(s)-\big(\frac{1}{s}+\delta(s)\big)\frac{\lambda}{s})-s+\frac{\lambda}{s}
\end{equation}
which, in light of the known asymptotics is of the form 
\begin{equation}
\lambda'=-(\lambda+1)s+o(s)
\end{equation}
from which we get
\begin{equation}\label{lambda_as}
\lim_{s\rightarrow \infty} \lambda(s)=-1.
\end{equation}
Now, writing $\lambda(s)=-1+\frac{\mu(s)}{s^2}$ we see, similarly to the above,  that $\mu$ is sub-quadratic, and satisfies the estimate
\[
\mu'(s)=-(\mu+2)s +o(s)
\]
which implies that 
\begin{equation}\label{mu_as}
\lim_{s\rightarrow \infty} \mu(s)=-2.
\end{equation}
Thus $\lambda(s)=-1+O(s^{-2})$, and so $\psi(s)=-\frac{1}{s}+O(s^{-3})$, thus giving
\begin{equation}
\phi(s)=s-\frac{1}{s}+O(s^{-3}).
\end{equation}
 
\end{proof}
\section{Proof of Thoerem \ref{main_thm_b}}
\begin{proof}[Proof of Theorem \ref{main_thm_b}]. Let $M_t$ be such a MCF, and let $\{(x_k,t_k)\}_{k=1}^{\infty}\rightarrow (x_0,t_0)$ with $t_k \leq t_0$, and $\lambda_k\rightarrow \infty$ be such that the flow
\begin{equation}\label{lim_flow}
\lambda_k(M_{\lambda_k^{-2}(s+t_k)}-x_k)\rightarrow N_s,
\end{equation}
where the limit is in the $C^{\infty}_{\mathrm{loc}}$ sense, and $N_s$ is a translating flow. Assume w.l.g that $N_s$ translates in the direction $\tau$ with velocity $1$, and note that 
\[
1<\mathcal{E}(N) \leq \mathcal{E}(M_0)<\infty,
\]
where $\mathcal{E}$ is the entropy functional (see \cite{CM}). Now, by Huisken's monotonicity formula \cite{Huisken_monotonicity} and \cite{Ilm}, there exists a sequence $\{s_j\}_{j=1}^{\infty}$ with $s_j\rightarrow -\infty$ such that 
\begin{equation}\label{ass_converge}
\lim_{j\rightarrow \infty} N_{s_j}/\sqrt{-s_j} \rightarrow \mathcal{S}\;\;\;\;\;\textrm{as varifolds},
\end{equation}
where $\mathcal{S}$ is a  smooth, non-compact self shrinker (potentially with multiplicity) of the MCF with $\mathcal{E}(\mathcal{S})=\mathcal{E}(N)$.
\begin{claim}
$\mathcal{S}$ is a multiplicity one cylinder.
\end{claim}
\begin{proof}
\eqref{lim_flow} and \eqref{ass_converge} together, imply that there exist a sequence of points in the original space time $(x'_k,t'_k)$ with $t'_k <t_0$, such that $(x'_k,t'_k)\rightarrow (x_0,t_0)$ and $r'_k\rightarrow 0$, such that for every $\eps>0$, for $k$ large enough,   
\begin{equation}
\frac{1}{4\pi {r'_k}^2}\int_{M_{t'_k-{r'_k}^2}}\exp(-|x-x'_k|^2/4{r'_k}^2) \geq \mathcal{E}(\mathcal{S})-\eps.
\end{equation}
On the other hand, since $(x_0,t_0)$ is a cylindrical singularity, there exist $r>0$ such that for every point in $(p,t)\in P((x_0,t_0),r)$
\begin{equation}
\frac{1}{4\pi r^2}\int_{M_{t-r^2}}\exp(-|x-p|^2/4r^2) \leq \mathcal{E}(S^1)+\eps.
\end{equation}
By Huisken monotonicity formula and by the arbitrariness of $\eps$, this implies that 
\[
\mathcal{E}(\mathcal{S}) \leq \mathcal{E}(S^1).
\]
By \cite[Corollary 1.2]{BW}, $\mathcal{S}$ is a multiplicity one cylinder ($\mathcal{S}$ is not the 2 sphere, as it is non-compact).
\end{proof}
Since $\mathcal{S}$ is a multiplicity one cylinder, by Brakke regularity theorem \cite{Bra,White_reg}, the convergence in \eqref{ass_converge} is in fact smooth. The result now follows from Theorem \ref{main_thm}.
\end{proof}
\bibliography{Trans}
\bibliographystyle{alpha}
\end{document}